\documentclass[12pt]{article}

\usepackage{diagbox}
\usepackage{mathtools}
\usepackage{bbm}
\usepackage{latexsym}
\usepackage{epsfig}
\usepackage{amsmath,amsthm,amssymb,enumerate,hyperref}
\usepackage{comment}
\usepackage[active]{srcltx}
\usepackage{color}
\def\red{\color{black}}

\newcommand{\LEM}[2]{\begin{lemma}\label{#1}#2\end{lemma}}

\newcounter{rot}

\def\a{\alpha}   
 \def\f{\phi}   \def\g{\gamma}
  
     \def\l{\lambda}
 \def\m{\mu}  
\def\r{\rho}  \def\s{\sigma} 
 \def\om{\omega}

\def\cP{{\cal P}}

\def\cB{{\cal B}}

\newtheorem{theorem}{Theorem}
\newtheorem{lemma}[theorem]{Lemma}

\newcommand{\wh}[1]{\widehat{#1}}

\newcommand{\rdup}[1]{{\left\lceil #1\right\rceil }}

\newcommand{\brac}[1]{\left(#1\right)}

\newcommand{\bfrac}[2]{\left(\frac{#1}{#2}\right)}

\def\cE{{\cal E}}

\newcommand{\set}[1]{\left\{#1\right\}}

\def\Pr{\mathbb{P}}

\newcommand{\ignore}[1]{}

\newcommand{\card}[1]{\left|#1\right|}

\newcommand{\beq}[2]{\begin{equation}\label{#1}#2\end{equation}}

\def\cG{\mathcal{G}}

\newcommand{\multstar}[1]{\begin{multline*}#1\end{multline*}}

\usepackage{tikz}
\usetikzlibrary{arrows}
\usetikzlibrary{decorations}
\usetikzlibrary{shapes.misc}

\begin{document}
\author{Debsoumya Chakraborti\thanks{Department of Mathematical Sciences, Carnegie Mellon University, Pittsburgh, PA15213, USA}, Alan Frieze\thanks{Department of Mathematical Sciences, Carnegie Mellon University, Pittsburgh, PA15213, USA, Research supported in part by NSF grant DMS1661063}, Simi Haber\thanks{Department of Mathematics, Bar-Ilan University. This research was partially supported by the BIU Center for Research in Applied Cryptography and Cyber Security in conjunction with the Israel National Directorate in the Prime Minister’s office.}, Mihir Hasabnis\thanks{Department of Mathematical Sciences, Carnegie Mellon University, Pittsburgh, PA15213, USA}}

\title{Isomorphism for Random $k$-Uniform Hypergraphs }
\maketitle

\begin{abstract}
We study the isomorphism problem for random hypergraphs. We show that it is solvable in polynomial time for the binomial random $k$-uniform hypergraph $H_{n,p;k}$, for a wide range of $p$. We also show that it is solvable w.h.p. for random $r$-regular, $k$-uniform hypergraphs $H_{n,r;k},r=O(1)$.
\end{abstract}

\section{Introduction}
In this note we study the isomorphism problem for two models of random $k$-uniform hypergraphs, $k\geq 3$. A hypergraph is $k$-uniform if all of its edges are of size $k$. The graph isomorphism problem for random graphs is well understood and in this note we extend some of the ideas to hypergraphs.

The first paper to study graph isomorphism in this context was that of Babai, Erd\H{o}s and Selkow \cite{BES}. They considered the model $G_{n,p}$ where $p$ is a constant independent of $n$. {\red They showed that w.h.p.\footnote{A sequence of events $\cE_n,n\geq 1$ occurs with high probability (w.h.p.) if $\lim_{n\to\infty}\Pr(\cE_n)=1$.} a {\em canonical labelling} of $G=G_{n,p}$  can be constructed in $O(n^2)$ time.} In a canonical labelling we assign a unique label to each vertex of a graph such that labels are invariant under isomorphism. It follows that two graphs with the same vertex set are isomorphic, if and only if the labels coincide. (This includes the case where one graph has a unique labeling and the other does not. In which case the two graphs are not isomorphic.) The failure probability for their algorithm was bounded by $O(n^{-1/7})$.  Karp \cite{K}, Lipton \cite{L} and Babai and Kucera \cite{BK} reduced the failure probability to $O(c^n),c<1$. These papers consider $p$ to be constant and the paper of Czajka and Pandurangan \cite{CP} allows $p=p(n)=o(1)$. We use the following result from \cite{CP}: the notation $A_n\gg B_n$ means that $A_n/B_n\to\infty$ as $n\to \infty$.
\begin{theorem}\label{th1}
Suppose that $p\gg\frac{\log^4n}{n\log\log n}$ and $p\leq \frac12$. Then {\red there is a polynomial time algorithm that finds a canonical labeling q.s.\footnote{A sequence of events $\cE_n,n\geq 1$ occurs quite surely (q.s.) if $\Pr(\cE_n)=1-O(n^{-K})$ for any positive constant $K$.} for $G_{n,p}$. In fact the running time of the algorithm is $O(n^2p)$ q.s.}
\end{theorem}
Our first result concerns the random hypergraph $H_{n,p;k}$, the random $k$-uniform hypergraph on vertex set $[n]$ in which each of the possible edges {\red in} $\binom{[n]}{k}$ occurs independently with probability $p$. We say that two $k$-uniform hypergraphs $H_1,H_2$ are isomorphic if there is a bijection $f:V(H_1)\to V(H_2)$ such that $\set{x_1,x_2,\ldots,x_k}$ is an edge of $H_1$ if and only if  $\set{f(x_1),f(x_2),\ldots,f(x_k)}$ is an edge of $H_2$. 
\begin{theorem} \label{th2}
Suppose that $k\geq 3$ and $p,1-p\gg n^{-(k-2)}\log n$ then {\red there exists an $O(n^{2k})$ time algorithm that finds a canonical labeling for $H_{n,p;k}$ w.h.p.}
\end{theorem}
Bollob\'as \cite{B} and Kucera \cite{Ku} proved that random regular graphs have canonical labelings w.h.p. We extend the argument of \cite{B} to regular hypergraphs. A hypergraph is regular of degree $r$ if every vertex is in exactly $r$ edges. We denote a random $r$-regular, $k$-uniform hypergraph on vertex set $[n]$ by $H_{n,r;k}$.
\begin{theorem} \label{th3}
{\red Suppose that $r,k$ are constants. Then there is an $O(n^{8/5})$ time algorithm that finds a canonical labeling for $H_{n,r;k}$ w.h.p.}
\end{theorem}
\section{Proof of Theorem \ref{th2}}
Given $H=H_{n,p;k}$ we let $H_i$ denote the $(k-1)$-uniform hypergraph with vertex set $[n]\setminus \set{i}$ and edges $\set{{\red e\setminus \set{i}:\;i\in e\in E(H)}}$. {\red $H_i$ is known as the {\em link} associated with vertex $i$.} Let $\cE_k$ denote the event $\set{\not\exists i,j:H_i\cong H_j}$.
\begin{lemma}\label{lem1}
Suppose that $k\geq 3$ and $\om\to\infty$ and $p,1-p\geq \om n^{-(k-2)}\log n$. Then $\cE_k$ occurs q.s.
\end{lemma}
\begin{proof}
\begin{align*}
\mathbb{P}(\exists i,j:H_i\cong H_j) &\leq n^4n! (p^2 + (1-p)^2)^{\binom{n-4}{k-1}}\\
&\leq 3n^{9/2}{\bfrac{n}{e}}^n (p^2+(1-p)^2)^{\binom{n-4}{k-1}p}\\
&\leq n^{-\om/k!}.
\end{align*}
{\bf Explanation:} There are $\binom{n}{2}$ choices for $i,j$. There are at most $n^2$ choices for $y=f(i),x=f^{-1}(j)$ in an isomorphism $f$ between $H_i$ and $H_j$. This accounts for the $n^4$ term. There are $(n-3)!<n!$ possible isomorphisms between $H_i-\set{y,j}$ and $H_j-\set{x,i}$. Then for every $(k-1)$-set of vertices $S$ that includes none of $i,j,x,y$, the probability for there to be an edge or non-edge in both $H_i$ and $H_j$ is given by the expression $p^2 + (1-p)^2$. 

The above estimation shows that even disregarding edges containing $i,j,x$ or $y$, w.h.p. there are no $i,j$ with $H_i \cong H_j$.
\end{proof}
Let $\cG_k$ be the event that {\red a canonical labeling for $H_{n,p;k}$} can be constructed in $O(n^{2k})$ time. Now assume inductively that
\beq{ind}{
\Pr(H_{n,p;k}\notin \cG_k)\leq n^{-\om/(k+1)!}.
}
The base case, $k=2$, for \eqref{ind} is {\red  given by the result of  \cite{L}, although in addition \cite{K}, \cite{CP} can be used for constant $p$}. Let $\cB_i$ be the event that $H_i\notin \cG_{k-1}$. Then
\beq{1}{
\Pr(H_{n,p;k}\notin \cG_k)\leq {\red (1-\Pr(\cE_k))}+\sum_{i=1}^n\Pr(\cB_i).
}
Indeed, if none of the events in \eqref{1} occur then in time $O(n^2\times n^{2(k-1)})=O(n^{2k})$ we can by induction uniquely label each vertex via the {\red labelled link}. After this we can confirm that $\cE_k$ has occurred. This confirms the claimed time complexity. Given that {\red $\cE_k$ has occured}, this will determine the only possible isomorphism {\red $\f$ between $H$ and any other $k$-uniform hypergraph $H'$ on vertex set $[n]$. We determine $\f$ by comparing the links of $H,H'$, using induction to see if they are isomorphic. We can see if there is a mapping $\f$ such that the link of $i$ in $H$ is isomorphic to the link of $\f(i)$ in $H'$ then and then we check to see whether or not $\f$ is actually an isomorphism.}

Going back to \eqref{1} we see by induction that
\[
\Pr(H_{n,p;k}\notin \cG_k)\leq n^{-\om/k!}+n^2\times (k-1)n^{2k-2}n^{-\om/k!}\leq n^{-\om/(k+1)!}.
\]
This completes the proof of Theorem \ref{th2}.
\section{Proof of Theorem \ref{th3}}
We extend the analysis of Bollob\'as \cite{B} to hypergraphs. For a vertex $v$, we let $d_\ell(v)$ denote the number of vertices at hypergraph distance $\ell$ from $v$ in $H=H_{n,r;k}$. We show that if 
\[
\ell^*=\rdup{\frac{3}{5}\log_{\r}n}\text{ where }\r=(r-1)(k-1).
\]
then w.h.p. no two vertices have the same sequence $(d_\ell(v),\ell=1,2,\ldots,\ell^*)$. By doing a breadth first search from each vertex of  $H$ we can therefore w.h.p. distinctly label each vertex within $O(n\r^{\ell^*})=O(n^{8/5})$ steps.

We use the {\red configuration} model for hypergraphs, which is a simple generalisation of the model in Bollob\'as \cite{B1}. We let {\red $W=[rn]$} where $m=rn/k$ is an integer. Assume that it is partitioned into sets $W_1,W_2,\ldots,W_n$ of size $r$. We define $f:W\to [n]$ by $f(w)=i$ if $w\in W_i$. A configuration $F$ is a partition of $W$ into sets $F_1,F_2,\ldots,F_m$ of size $k$. Given $F$ we obtain the (multi)hypergraph $\g(F)$ where $F_i=\set{w_1,w_2,\ldots,w_k}$ gives rise to the edge $\set{f(w_1),f(w_2),\ldots,f(w_k)}$ for $i=1,2,\ldots,m$.
{\red Configurations can contain multiple edges and loops. Nevertheless, it is known that if $\g(F)$ has a hypergraph property w.h.p. then $H_{n,r;k}$ will also have this property w.h.p.,} see for example \cite{CFMR}. 

In the following $H=H_{n,r;k}$. For a set $S\subseteq [n]$, we let $e_H(S)$ denote the number of edges of $H$ that are contained in $S$.
\LEM{lemiso1}{
Let $\ell_0=\rdup{100\log_{\r}\log n}$. Then w.h.p., $e_H(S)<\frac{|S|+1}{k-1}$ for all $S\subseteq [n], |S|\leq {\red 10\ell_0}$.
}
\begin{proof}
We have that 
\begin{align}
&\mathbb{P}\brac{\exists S:|S|\leq 10\ell_0,e_H(S)\geq \frac{|S|+1}{k-1}}\nonumber\\
&\leq \sum_{s=4}^{10\ell_0}\binom{n}{s}\binom{sr}{\frac{s+1}{k-1}} \bfrac{\binom{sr}{k-1}}{\binom{km-10k\ell_0}{k-1}}^\frac{s+1}{k-1}\label{erf1}\\
&\leq  \sum_{s=4}^{10\ell_0}\bfrac{ne}{s}^s(er(k-1))^{\frac{s+1}{k-1}}\bfrac{rs}{rn-o(n)}^{s+1}\nonumber\\
&\leq\frac{1}{n^{1-o(1)}}\sum_{s=4}^{10\ell_0}se^s\brac{e(k-1)r}^{\frac{s+1}{k-1}} =o(1).\nonumber
\end{align}
{\red
{\bf Explanation for \eqref{erf1}:} we choose a set $S$ and then a set $X$ of $(s+1)/(k-1)$ members of $W_S=\bigcup_{i\in S}W_i$. We then estimate the probability that each member of $X$ is paired in $F$ with $k-1$ members of $W_S\setminus X$. For each $x\in X$, given some previous choices, there are at most $\binom{sr}{k-1}$ choices contained in $W_S$, out of at least $\binom{km-10k\ell_0}{k-1}$ choices overall.}
\end{proof}
Let $\cE$ denote the high probability event in Lemma \ref{lemiso1}. We will condition on the occurrence of $\cE$.

Now for $v\in [n]$, let $S_\ell(v)$ denote the set of vertices at distance $\ell$ from $v$ and let $S_{\leq \ell}(v)=\bigcup_{j\leq \ell}S_{j}(v)$. {\red Here the distance between vertices $u,v$ is the minimum length of a path/sequence of edges $e_1,e_2,\ldots,e_k$ such that $u\in e_1,v\in e_k$ and $e_i\cap e_{i+1}\neq \emptyset$ for $1\leq i<k$.} We note that
{\red
\beq{smallSk}{
|S_\ell(v)|\leq (k-1)r\r^{\ell-1}\text{ for all }v\in [n],\ell\geq 1.
}}
Furthermore, Lemma \ref{lemiso1} implies that there exist $b_{r,k}<a_{r,k}<(k-1)r$ such that w.h.p., we have for all $v,w\in [n], 1\leq \ell\leq \ell_0$, 
\begin{align}
|S_\ell(v)|&\geq a_{r,k}\r^{\ell-1}.\label{smallSk1}\\
|S_\ell(v)\setminus S_\ell(w)|&\geq b_{r,k}\r^{\ell-1}.\label{smallSk1a}
\end{align}
{\red 
To see this, observe that $|S_{\ell+1}|=\r|S_{\ell}|$ unless there are two vertices $x,y\in S_\ell$ and either (i) an edge $e$ of $\g(F)$ such that $e\supseteq \set{x,y}$ or (ii) a vertex $z\in S_{\ell+1}$ and edges $e,f$ of $\g(F)$ such that $e\supseteq \set{x,z}$ and $f\supseteq \set{y,z}$. Lemma \ref{lemiso1} implies that w.h.p. there is at most one such case of (i) or (ii) for $1\leq \ell\leq \ell_0$. Suppose that there are two distinct edges $e_i,i=1,2$ that cause (i) at levels $\ell_1,\ell_2$ and suppose that $\set{x_i,y_i}$ corresponds to $e_i,i=1,2$. Each $x\in S_\ell$ lies in the final edge of an $\ell$-length path $P_u$ from $v$ to $x$. Now $\cP=P_{x_1}, P_{x_2},P_{y_1},P_{y_2}$ spans $\Pi\leq 2(\ell_1+\ell_2)\leq 2\ell_0$ edges and we can choose these paths to not contain $e_1$ or $e_2$. Furthermore, $\cP$ spans at most $1+(k-1)\Pi$ vertices, since adding a new edge to a connected set of vertices adds at most $k-1$ new vertices. If we add $e_1,e_2$ to these $\Pi$ edges then we have at most $1+(k-1)\Pi+2(k-2)$ vertices spanning $\Pi+2$ edges and this contradicts Lemma \ref{lemiso1}. The remaining two cases (2 times (ii) or (i) and (ii)) can be argued similarly.
So, typically adding an edge in the construction of $S_{\ell+1}$ adds $k-1$ new vertices. W.h.p., there is one case and this only adds $k-2$ vertices. This explains \eqref{smallSk1}. 

A similar argument yields \eqref{smallSk1a}. Having constructed $S_{\ell}(w)$, we see that typically adding an edge  in the construction of $S_{\ell}(v)$ adds $k-1$ new vertices to the union  $S_{\ell}(v)\cup  S_{\ell_0}(w)$ and that w.h.p. it adds at least $k-2$ vertices.
}

Now consider $\ell>\ell_0$. Consider doing breadth first search from $v$ or $v,w$ exposing the configuration pairing as we go. Let an edge be {\em dispensable} if it contains two vertices already known to be in $S_{\leq \ell}$. {\red The argument above} implies that w.h.p. there is at most one dispensable edge in $S_{\leq\ell_0}$. 
\LEM{lemiso2}{
With probability $1-o(n^{-2})$, (i) at most 20 of the first $n^{\frac{2}{5}}$ exposed edges are dispensable and (ii) at most $n^{1/4}$ of the first $n^{\frac{3}{5}}$ exposed edges are dispensable.
}
\begin{proof}
The probability that the $\s$th edge is dispensable is at most \\{\red $\frac{r((\s-1)(k-1)+1)(k-1)}{rn-k\s}$}, independent of the history of the process. {\red (Knowing one vertex of this edge and choosing the rest of it, there are at most  $r((\s-1)(k-1)+1)(k-1)$ choices out of at least $rn-k\s$ that will lead to this edge being dispensable.)} Hence,
\multstar{
\mathbb{P}(\exists\; 20\text{ dispensable edges in the first }n^{2/5})\leq \binom{n^{2/5}}{20}\bfrac{{\red rk^2}n^{2/5}}{rn-o(n)}^{20}\\
=o(n^{-2}).
}
\multstar{
\mathbb{P}(\exists\; n^{1/4}\text{ dispensable edges in first }n^{3/5})\leq \binom{n^{3/5}}{n^{1/4}}\bfrac{{\red rk^2}n^{3/5}}{rn-o(n)}^{n^{1/4}}\\
=o(n^{-2}).
}
\end{proof}
Now let {\red $\ell_1=\rdup{\log_{r-1}n^{2/5}}$ and $\ell_2=\rdup{\log_{r-1}n^{3/5}}$}. Then, we have that, conditional on $\cE$, with probability $1-o(n^{-2})$,
\begin{align*}
&|S_\ell(v)|\geq (a_{r,k}\r^{\ell_0-1}-40)\r^{\ell-\ell_0}:\;\ell_0<\ell\leq\ell_1.\\
&|S_\ell(v)|\geq (a_{r,k}\r^{\ell_1-1}-40\r^{\ell_1-\ell_0}-2n^{1/4})\r^{\ell-\ell_1};\; \ell_1<\ell\leq\ell_2.\\
&|S_\ell(w)\setminus S_\ell(v)|\geq (b_{r,k}\r^{\ell_0-1}-40)\r^{\ell-\ell_0}:\;\ell_0<\ell\leq\ell_1.\\
&|S_\ell(w)\setminus S_\ell(v)|\geq (b_{r,k}\r^{\ell_1-1}-40\r^{\ell_1-\ell_0}-2n^{1/4})\r^{\ell-\ell_1};\; \ell_1<\ell\leq\ell_2.
\end{align*}
We deduce from this that if $\ell_3=\rdup{\log_{r-1}n^{4/7}}$ and $\ell=\ell_3+a,a=O(1)$ then with probability $1-o(n^{-2})$,
\begin{align*}
|S_\ell(w)|&\geq (a_{r,k}-o(1))\r^{\ell-1}\approx a_{r,k}\r^{a-1}n^{4/7} .\\
|S_\ell(w)\setminus S_\ell(v)|&\geq (b_{r,k}-o(1))\r^{\ell-1}\approx b_{r,k}\r^{a-1}n^{4/7} .
\end{align*}
Suppose now that we consider the execution of breadth first search up until we have exposed {\red $S_{\ell+1}(v)\cup S_{\ell+1}(w)$ and the edges $\wh{E}$ defining this set. We let $U=W\setminus \wh{E}$. We will show that we can find a position in the process so that conditioning up to this point, in order to have $|S_{\ell+1}(v)|=|S_{\ell+1}(w)|$ there will have to be a prescribed, but unlikely, outcome for a large number of edge selections. 

Our conditioning includes all the choices of $e\in F$ that are necessary to construct $S_{\ell+1}(v)\cup S_{\ell}(w)$. We refer to a choice of $e$ as an {\em edge-selection}. After an edge-selection $e$, we update $U\gets U\setminus \set{e}$. Consider the edge-selections involving $W_x,x\in S_\ell(w)\setminus S_{\ell+1}(v)$. Now at most $n^{1/4}$ of these edge-selections involve vertices in $S_{\leq \ell+1}(v)\cup S_{\leq \ell}(w)$. Condition on these as well.  There must now be $\l=\Theta(n^{4/7})$ further  edge-selections containing elements of $W_x,x\in S_\ell(w)\setminus S_{\ell+1}(v)$ and $W_y,y\notin S_{\ell+1}(v)\cup S_\ell(w)$. Let $Z$ denote the vertices in $S_\ell(w)$ involved in these $\l$ edge-selections. Furthermore, to have  $|S_{\ell+1}(v)|=|S_{\ell+1}(w)|$ these $\l$ selections must involve exactly $t$ of the sets  $W_y,y\notin S_{\ell+1}(v)\cup S_\ell(w)$. Here $t$ is the unique value that will ensure that $|S_{\ell+1}(w)|=|S_{\ell+1}(v)|$. The important point is that $t$ is determined {\em before} the making of these $\l$ edge-selections. Let $R=\bigcup_{y\notin S_{\ell+1}(v)\cup S_\ell(w)}W_y$ at the point immediately prior to the $\l$ edge-selections. Let $S=R\cap \bigcup_{e:e\cap Z\neq \emptyset}e$ and note that $S$ is a random $s$-subset of $R$ for some $s=\Theta(n^{4/7})$.

The following lemma will easily show that w.h.p. $H$ has a canonical labeling defined by the values of $|S_{\ell}(v)|, 1\leq \ell\leq \ell^*,\,v\in [n]$.
}
\LEM{isomain}{
Let $R=\bigcup_{i=1}^\m R_i$ be a partitioning of an $r\m$ set $R$ into $\m$ subsets of size $r$. Suppose that $S$ is a random $s$-subset of $R$, where $\m^{5/9}<s<\m^{3/5}$. Let $X_S$ denote the number of sets $R_i$ intersected by $S$. Then
\[
\max_j\mathbb{P}(X_S=j)\leq \frac{c_0\m^{1/2}}{s},
\]
for some constant $c_0$.
}
\begin{proof}
The probability that $S$ has at least 3 elements in some set $R_i$ is at most
\[
\frac{\m\binom{r}{3}\binom{r\m-3}{s-3}}{\binom{r\m}{s}}\leq \frac{s^3}{6\m^2}\leq \frac{\m^{1/2}}{6s}.
\]
But
\[
\mathbb{P}(X_S=j)\leq \mathbb{P}\brac{\max_i|S\cap R_i|\geq 3}+\mathbb{P}\brac{X_S=j\text{ and }\max_i|S\cap R_i|\leq 2}.
\]
So the lemma will follow if we prove that for every $j$,
\beq{lemfollow}{
P_j=\mathbb{P}\brac{X_S=j\text{ and }\max_i|S\cap R_i|\leq 2}\leq \frac{c_1\m^{1/2}}{s},
}
for some constant $c_1$.

Clearly, $P_j=0$ if $j<s/2$ and otherwise
\beq{defPj}{
P_j=\frac{\binom{\m}{j}\binom{j}{s-j}r^{2j-s}\binom{r}{2}^{s-j}}{\binom{r\m}{s}}.
}
Now for $s/2\leq j<s$ we have
\beq{Pjrat}{
\frac{P_{j+1}}{P_j}=\frac{(\m-j)(s-j)}{(2j+2-s)(2j+1-s)}\frac{2r}{r-1}.
}
We note that if $s-j\geq \frac{10s^2}{\m}$ then $\frac{P_{j+1}}{P_j}\geq \frac{10r}{3(r-1)}\geq 2$ and so the $j$ maximising $P_j$ is of the form $s-\frac{\a s^2}{\m}$ where $\a\leq 10$. If we substitute $j=s-\frac{\a s^2}{\m}$ into \eqref{Pjrat} then we see that
\[
 \frac{P_{j+1}}{P_j}\in \frac{2\a r}{r-1}\left[1\pm c_2\frac{s}{\m}\right]
\]
for some absolute constant $c_2>0$.

It follows that if $j_0$ is the index maximising $P_j$ then
\[
\card{j_0-\brac{s-\frac{(r-1)s^2}{2r\m}}}\leq 1.
\]
Furthermore, if $j_1=j_0-\frac{s}{\m^{1/2}}$ then 
\[
 \frac{P_{j+1}}{P_j}\leq 1+c_3\frac{\m^{1/2}}{s}\text{ for }j_1\leq j\leq j_0,
\]
for some absolute constant $c_3>0$.

This implies that for all $j_1 \le j \le j_0$,
\multstar{
P_j\geq P_{j_0}\brac{1+c_3\frac{\m^{1/2}}{s}}^{-(j_0-j_1)}\\= P_{j_0}\exp\set{-(j_0-j_1)\brac{c_3\frac{\m^{1/2}}{s}+O\bfrac{\m}{s^2}}}
\geq P_{j_0}e^{-2c_3}.
}
It follows from this that 
\[
P_{j_0} \le e^{2c_3} \min_{j \in [j_1,j_0]} P_j \le \frac{e^{2c_3}}{j_0 - j_1} \sum_{j \in [j_1,j_0]} P_j \le \frac{e^{2c_3}\m^{1/2}}{s}.
\]
\end{proof}
We apply Lemma \ref{isomain} with $\m=n-o(n),s=\Theta(n^{4/7})$ to show that 
\[
\mathbb{P}((|S_\ell(v)|=|S_\ell(w)|,\ell\in [\ell_3,\ell_3+14])\leq \bfrac{c_0n^{1/2}}{n^{4/7}}^{15}=o(n^{-2}).
\]
This completes the proof of Theorem \ref{th3}.

\end{document}